\documentclass[12pt]{amsart}

\usepackage{amsfonts,amsmath,amssymb}

\newtheorem{Lem}{Lemma}[section]
\newtheorem{Prop}[Lem]{Proposition}
\newtheorem{Char}[Lem]{Characterization}
\newtheorem{Cor}[Lem]{Corollary}
\newtheorem{Thm}[Lem]{Theorem}

\theoremstyle{definition}

\newtheorem{Def}[Lem]{Definition}
\newtheorem{Rem}[Lem]{Remark}

\newtheorem{Expl}[Lem]{Example}

\newcommand\pf{\begin{proof}}
\newcommand\epf{\end{proof}}

\def\NZQ{\Bbb}               
\def\NN{{\NZQ N}}

\def\ZZ{{\NZQ Z}}

\newcommand\Sym{{\operatorname{Sym}}}

\newcommand\supp{{\operatorname{supp}}}
\newcommand\reg{{\operatorname{reg}}}
\newcommand\depth{{\operatorname{depth}}}
\newcommand\proj{{\operatorname{proj}}}

\newcommand\e{{\operatorname{e}}}
\newcommand\m{{\operatorname{m}}}

\numberwithin{equation}{section}

\title{Algebraic properties of universal squarefree lexsegment ideals $ ^1$}

\author{Marilena Crupi
}
\address{University of Messina, Department of Mathematics and Computer Science\\
Viale Ferdinando Stagno d'Alcontres, 31 \\
98166 Messina, Italy} \email{mcrupi@unime.it}
\author{Monica La Barbiera}

\address{University of Messina, Department of Mathematics and Computer Science\\
Viale Ferdinando Stagno d'Alcontres, 31\\
98166 Messina, Italy} \email{monicalb@unime.it}
\begin{document}

\thanks{\textbf{$^1$ To appear in Algebra Colloquium}}
\maketitle

\begin{abstract}
Let $K$ be a field and let $A=K[X_1, \ldots, X_n]$ be the
polynomial ring in $X_1, \ldots, X_n$ with coefficients in the
field $K$. We study the universal squarefree lexsegment ideals. We put our attention on their combinatorics computing some invariants. Moreover we study the link between such special class of squarefree lexsegment ideals and the so called $s$-sequences.\\

\end{abstract}

\section*{Introduction}\label{sec0}
Let $K$ be a field and let $A=K[X_1, \ldots, X_n]$ be the
polynomial ring in $X_1, \ldots, X_n$ with coefficients in the
field $K$. Set $A_{[m]} = K[X_1, \cdots, X_n,X_{n+1},\cdots, X_{n+m}]$, where $m$ is a positive integer.
A squarefree lexsegment ideal $I$ of $A$ is called \textit{universal squarefree lexsegment ideal}
(abbreviated USLI), if for any integer $m\geq 1$, the squarefree monomial ideal $IA_{[m]}$
of the polynomial ring $A_{[m]}$ is a squarefree lexsegment ideal.
 Let $M^s$ denote the set of all squarefree monomials in the variables $X_1, \cdots, X_n$.
 A squarefree lexsegment ideal $I$ of $A$
 with $G(I) = \{u_1, \ldots, u_{\ell}, u_{{\ell}+1}\}$, $u_1> \ldots> u_{\ell}> u_{{\ell}+1}$
 with respect to the homogeneous lexicographic order on $M^s$, is called
 \textit{almost universal squarefree lexsegment ideal} (abbreviated AUSLI),
 if $I$ is not an USLI of $A$ but the ideal $J = (u_1, \ldots, u_{\ell})$ is an USLI of $A$.
 $G(I)$ is the unique minimal set of monomial generators of the monomial ideal $I$. These definitions were introduced by Babson, Novik and Thomas in \cite{BNT} in order to study the simmetric version of algebraic shifting. The algebraic shifting is an algebraic operation introduced by Kalai (\cite{BK}, \cite{GK}) that transforms
  a simplicial complex into a simpler complex that preserves important combinatorial, topological and algebraic invariants.

In this paper we put our attention on the structure of universal squarefree lexsegment ideals (Characterization \ref{char}). We analyze their combinatorics in order to compute some invariants as the projective dimension, the Castelnuovo-Mumford regularity and the depth (Corollary \ref{inv}).

In \cite{HRT}, the authors introduced the concept of $s$-sequences
in order to study the symmetric algebra of a module $M$ on a
noetherian ring $R$. One of their motivation was that is a difficult
problem to compute standard algebraic invariants of the graded
algebra $\Sym_R(M)$. Their proposal was to determine these invariants
in terms of the corresponding invariants of
special quotients of the ring $R$. The $s$-sequences are an
important tool for this computation.

In this paper we analyze the problem when a squarefree lexsegment
ideal $I$ of degree $d$ of the polynomial ring $A$  is generated by
an $s$-sequence. We are able to state that this happens if $I$ is an
USLI or an AUSLI (Theorem \ref{T*}). Consequently their  symmetric algebra is studied (Theorems \ref{inv:USLI} and \ref{inv:ASLI}). \\
The structure of the paper is organized as follows.\\
In section \ref{sec1}, we recall some notions that we will use  during the
paper. \\
In section \ref{sec2}, we describe in a suitable way the USLIs (Characterization \ref{char}). Hence we state a characterization of an USLI of degree $d$ (Proposition \ref{samedeg}). Moreover, we analyze some invariants associated to the universal squarefree lexsegment ideals. The main result states that an USLI $I\varsubsetneq A$ has a unique extremal Betti number whose value is $1$ (Proposition \ref{exU}). This fact allows us to compute $\proj_A(I)$, $\reg_A(I)$ and $\depth_A(A/I)$ (Corollary \ref{inv}). \\
Section \ref{sec3} is dedicated to the symmetric algebra of USLIs and AUSLs
of the polynomial ring $A$. More precisely let $I$ be a lexsegment
ideal of $A$ generated by squarefree monomials in a same degree,
we establish that the ideal $I$ is generated by an $s$-sequence if
and only if $I$ is an USLI or an AUSLI (Theorem \ref{T*}). This result is proved using the characterization of
  the monomial $s$-sequences  by the Gr\"{o}bner bases. As a consequence of this result  we study the
problem  of computing standard algebraic invariants of the graded
algebra $\Sym_A(I)$ when $I$ is a squarefree lexsegment ideal
generated by an $s$-sequence. More precisely, we give a formula for
the dimension and the multiplicity of $\Sym_A(I)$ when $I$ is an
AUSLI (Theorem \ref{inv:ASLI}). A formula for the
Castelnuovo-Mumford regularity and for the depth of $\Sym_A(I)$ when
$I$ is an USLI (Theorem \ref{inv:USLI}) is also stated.

\section{Preliminaries and notations}\label{sec1}

Let $K$ be a field and let $A=K[X_1, \ldots, X_n]$ be the
polynomial ring in $X_1, \ldots, X_n$ with coefficients in the
field $K$. We consider $A$ as an $\NN$-graded ring and each
deg$X_i$ = $1$. We denote by $M_d$ the set of all monomials of
degree $d$ of the polynomial ring $A$. If $I\subsetneq A$ is a monomial ideal we denote by $G(I)$ the unique minimal set of monomial generators of $I$ and by $G(I)_d$ the set $G(I)_d = \{u \in G(I)\,:\, \deg u = d\}$, for $d>0$.

For a monomial $1 \neq u \in A$, we set
\[\supp(u) = \{i \,:\, X_i \,\, \textrm{divides} \,\, u\}.\]
\[\m(u) = \max\{i : X_i \,\, \textrm{divides}\,\, u\}.\]
\par
Recall that a squarefree monomial ideal $I$ of $A$ is called
\textit{squarefree stable} if
for all $u\in G(I)$,
one has $(X_ju)/X_{\m(u)} \in I$ for all $j <\m(u)$
with $j \notin \supp(u)$
(\cite{Ahh:square}).

Now let $M^s_d$ denote the set of all squarefree monomials of degree $d\geq 1$ in the variables $X_1, \cdots, X_n$.
We write $>_{\textrm{slex}}$ for the \textit{lexicographic order} on the finite set $M^s_d$, that is,
 if $u = X_{i_1}\cdots X_{i_d}$ and $v = X_{j_1}\cdots X_{j_d}$ are squarefree monomials  belonging
 to $M^s_d$ with $1 \leq i_1 < i_2 < \cdots < i_d \leq n$ and
 $1 \leq j_1 < j_2 < \cdots < j_d \leq n$, then $u>_{\textrm{slex}} v$
 if $i_1 = j_1$, $\ldots$, $i_{s-1} = j_{s-1}$ and $i_s < j_s$ for some $1 \leq s \leq d$ (\cite{Ahh:square}).\\
Let $M^s$ be the set of all squarefree monomials in the variables $X_1, \cdots, X_n$.
We denote by $>_{\textrm{hslex}}$ the \textit{homogeneous lexicographic order} on $M^s$, that is,
if $u,v \in M^s$, then $u >_{\textrm{hslex}} v$ if $\deg u > \deg v$,
or if $\deg u = \deg v$ and $u>_{\textrm{slex}} v$.\\
 A monomial ideal $I \subsetneq A$ is called a \textit{squarefree lexsegment ideal} if $I$ is generated by squarefree monomials, and
for all squarefree monomials $u \in I$ and all squarefree monomials $v\in A$  with $\deg u = \deg v$ and
$v >_{\textrm{slex}} u$, then $v \in I$. Every squarefree lexsegment ideal of $A$ is
obviously a squarefree stable ideal.
\par\medskip
Set $A_{[m]} = K[X_1, \cdots, X_n,X_{n+1},\cdots, X_{n+m}]$, where $m$ is a positive integer.

We quote the next definitions from \cite{BNT}.

\begin{Def}  A squarefree lexsegment ideal $I$ of $A$ is called \textit{universal squarefree lexsegment ideal} (USLI, for short), if for any integer $m\geq 1$, the squarefree monomial ideal $IA_{[m]}$ of the polynomial ring $A_{[m]}$ is a squarefree lexsegment ideal.
\end{Def}
In other words a universal squarefree lexsegment ideal of $A$ is a squarefree lexsegment ideal $I$ of $A$ which remains being squarefree lexsegment if we regard $I$ as an ideal of the polynomial ring $A_{[m]}$ for all $m \geq 1$.

\begin{Expl}\label{expls}
%

(1) The squarefree lexsegment ideal $I = (X_1X_2, X_1X_3X_4)$ of $A = K[X_1,X_2,X_3,X_4]$ is an USLI. Indeed $I$ is a squarefree lexsegment ideal of the polynomial ring $A_{[m]}$ for all $m \geq 1$.\\
(2) The squarefree lexsegment ideal $I = (X_1X_2, X_1X_3X_4, X_2X_3X_4)$ of $A = K[X_1,X_2,$\newline $X_3,X_4]$ is not an USLI. Indeed $I$ is not a squarefree lexsegment ideal of the polynomial ring $A_{[1]} = K[X_1,X_2,X_3,X_4,X_5]$. In fact $X_1X_4X_5 >_{\textrm{slex}} X_2X_3X_4$ and
$X_1X_4X_5 \notin IA_{[1]}$.
\end{Expl}

\begin{Def}  A squarefree lexsegment ideal $I$ of $A$ with $G(I) = \{u_1, \ldots, u_\ell, u_{\ell+1}\}$, $u_1>_{\textrm{hslex}} \ldots>_{\textrm{hslex}} u_{\ell}>_{\textrm{hslex}} u_{{\ell}+1}$, is called \textit{almost universal squarefree lexsegment ideal} (AUSLI, for short), if $I$ is not an USLI of $A$ but the ideal $J = (u_1, \ldots, u_{\ell})$ is an USLI of $A$.
\end{Def}
\begin{Expl}\rm  The squarefree lexsegment ideal $I = (X_1X_2, X_1X_3X_4, X_2X_3X_4)$ of $A = k[X_1,X_2,X_3,X_4]$ is an AUSLI. Indeed $I$ is not an USLI of $A$, but the ideal $J = (X_1X_2, X_1X_3X_4)$ is an USLI of $A$ (Example \ref{expls}, (1)).
\end{Expl}
\par\medskip
We finish this section recall the notion of extremal Betti numbers of a graded ideal $I$ of the polynomial ring $A$.

If $I$ is a graded ideal of $A$, then $I$ has a
minimal graded free $A$-resolution
\[
F. : 0 \rightarrow F_s \rightarrow \cdots \rightarrow F_1
\rightarrow F_0 \rightarrow I \rightarrow 0
\]
where $F_i = \oplus_{j \in \ZZ}A(-j)^{\beta_{i,j}}$. \par The
integers $\beta_{i,j} = \beta_{i,j}(I) = \textrm{dim}_K
\textrm{Tor}_i(K, I)_j $ are called the graded Betti numbers of
$I$, while $\beta_i(I) = \sum_j \beta_{i,j}(I)$ are called the
total Betti numbers of $I$.

\par
To a graded ideal $I$ two invariants can be associated the
{\it projective dimension} and the {\it Castelnuovo-Mumford
regularity} (\cite{BH}, \cite{Eis}) that are defined, respectively,  as follows:
\[\proj_A(I) = \max\{i : \beta_i(I) \neq 0\},\]
\[
\begin{array}{lll}
\textrm{reg}_A(I) = \max\{j-i : \beta_{i, j}(I) \neq 0\}  = \max\{j
: \beta_{i,i+j}(I) \neq 0,\,\, \mbox{for some $i\in \NN$}\}.
\end{array}
\]

Bayer, Charalambous and Popescu introduced in \cite{BCP} a
refinement of the invariants above defined, giving the notion of
extremal Betti numbers.

\par
\begin{Def} \label{def:extr} A Betti number $\beta_{k,k+\ell}(I) \neq 0$ is called {\it
extremal} if $\beta_{i, i+j}(I) = 0$ for all $i \geq k$, $j \geq
\ell$, $(i, j) \neq (k, \ell)$.
\end{Def}
\par\medskip
If $\beta_{k_1,k_1+\ell_1}(I), \dots,
\beta_{k_t,k_t+\ell_t}(I),\,\,\,k_1 > \dots > k_t, \ell_1 <\dots <
\ell_t$, are all extremal Betti numbers of $I$, then $k_1 =
\proj_A(I)$ and $\ell_t = \textrm{reg}_A(I)$.

The following characterization of the extremal Betti numbers of squarefree stable ideals was given in \cite[Proposition 4.1]{CU2}.
\begin{Prop} \label{pro:extr} Let $I\subsetneq A$ be a squarefree stable ideal. The following conditions are equivalent:
\begin{enumerate}
\item $\beta_{k,\,k+\ell}(I)$ is extremal.
\item $k+\ell = \max\{\m(u)\,:\, u \in G(I)_{\ell}\}$ and $\m(u) < k +j$ for all $j>\ell$ and $u \in G(I)_j$.
\end{enumerate}
\end{Prop}

As a consequence of the above result, we obtain the following.
\begin{Cor} \label{value} Let $I\subsetneq A$ be a squarefree stable
ideal.
\begin{enumerate}
\item If $\beta_{k, \, k+ \ell}(I)$ is an extremal Betti
number of $I$, then
\[\beta_{k,\,k+\ell}(I) = \vert \{u \in G(I)_{\ell}\,:\,\m(u) = k +\ell\}\vert.\]
\item Set $d = \max\{j\,:\, G(I)_j\neq \emptyset\}$ and $m = \max\{\m(u)\,:\,u\in G(I)\}$,
then $\beta_{m - d,\,m - d+d}(I)$ is the unique extremal Betti
number of $I$ if and only if $m=\max\{\m(u)\,:\,u\in G(I)_d\}$ and
for every $w\in G(I)_j$, $j<d$, $\m(w)<m$.
\end{enumerate}
\end{Cor}


\section{Universal squarefree lexsegment ideals}\label{sec2}
In this section we discuss the combinatorics of universal squarefree
lexsegment ideals. Moreover we compute some standard invariants.

In \cite[Definition 4.1]{BNT}   there is a
characterization of USLIs. In order to reformulate it for our
purpose, we need to introduce some notations.

For a sequence of non negative integers $(k_i)_{i\in \NN}$, we define the following set:
\[\supp (k_i)_{i\in \NN} = \{i \in \NN\,:\, k_i \neq 0\}.\]

If $\supp (k_i)_{i\in \NN} = \{d_1, \ldots, d_t\}$, with $d_1 < d_2 < \ldots < d_t$, then we associate to $(k_i)_{i\in \NN}$ the following integers:
\[R_j = j + \sum_{i=1}^j k_i\]
for $j= 1,\ldots, d_t$. We set $R_j = 0$, for $j >d_t$.


Hence we can reformulate the characterization contained in \cite[Definition 4.1]{BNT}, as follows:

\begin{Char}\label{char} Let $I\subsetneq A$ be an ideal generated in degrees $d_ 1< d_2 < \ldots < d_t$. Then $I$ is an USLI of $A$ if and only if
\[\mbox{$G(I)_{d_i} = \left\{\left(\prod_{j=1}^{d_i-1}X_{R_j}\right)X_{\ell}\,:\, \ell \in [R_{d_i-1} +1, R_{d_i}-1]\right\}$, for \, $i=1, \ldots, t$,}\]
where $R_j = j + \sum_{i=1}^j \vert G(I)_{d_i}\vert$, for $j= 1,\ldots, d_t$.
\end{Char}
The characterization above follows from the statement contained in \cite[Definition 4.1]{BNT}, choosing $(k_i)_{i\in \NN}$ as the sequence of non negative integers such that $\supp (k_i)_{i\in \NN} = \{d_1, \ldots, d_t\}$ and $k_{d_i} = \vert G(I)_{d_i}\vert$, for $i = 1, \ldots, t$.
\begin{Rem} \label{seq} Assume that $(k_i)_{i\in \NN}$ is a sequence of non negative integers such that
\[\supp (k_i)_{i\in \NN} = \{d_1, \ldots, d_t\},\quad d_ 1< d_2 < \ldots < d_t.\]
Then there exists an USLI $I\subsetneq A= K[X_1, \ldots, X_n]$ generated in degrees $d_1, \ldots, d_t$ such that
$\vert G(I)_{d_i}\vert = k_{d_i}$, for $i = 1, \ldots, t$ if and only if $n \geq d_t + \sum_{i=1}^{d_t}k_i-1$.
\end{Rem}

Thanks to the above statements we can give the following characterizations of an USLI generated in a same degree $d$.
\begin{Prop} \label{samedeg} Let $I$ be a squarefree lexsegment ideal of $A = K[X_1, \ldots, X_n]$ generated in degree $d$. Then $I$ is an USLI  of $A$ if and only if $\vert G(I)\vert\leq n-d+1$.
\end{Prop}
\begin{proof} Let $(k_i)_{i\in \NN}$ be the sequence of non negative integers such that $\supp (k_i)_{i\in \NN} = \{d\}$, with $k_{d} = \vert G(I)_{d}\vert = \vert G(I)\vert$. From Remark \ref{seq}, $I\subsetneq A$ is an USLI generated in degree $d$ if and only if $n \geq d + k_d -1 = d + \vert G(I)\vert-1$ that is  if and only if $\vert G(I)\vert\leq n-d+1$.

\end{proof}

%
\par\medskip
Note that if $I\subsetneq A = K[X_1, \ldots, X_n]$ is an USLI generated in degree $d$ then
\begin{equation}\label{degree1}
    G(I) = \{X_1X_2\cdots X_{d-1}X_d, X_1X_2\cdots X_{d-1}X_{d+1}, \ldots, X_1X_2\cdots X_{d-1}X_k\},
\end{equation}
with $d \leq k \leq n$.

Moreover if $I$ is an AUSLI generated in degree $d$, then
\begin{equation}\label{degree2}
G(I) = \{X_1X_2\cdots X_{d-1}X_d, \ldots, X_1X_2\cdots X_{d-1}X_n,X_1X_2\cdots X_{d-2}X_dX_{d+1}\}.
  \end{equation}
\begin{Rem} \label{ASLI} \rm It is clear that a squarefree lexsegment ideal $I\varsubsetneq A$ generated in degree $d$ is an AUSLI if and only if $\vert G(I) \vert = n-d+2$.
\end{Rem}
\par\medskip
We finish this section computing some invariants of an USLI $I$ by its extremal Betti numbers. In general squarefree lexsegment ideals may have more than just one extremal Betti number (\cite{CU1},\cite{CU2}).\\
Consider, for example the squarefree ideal \[I=(X_1X_2,X_1X_3,X_1X_4,X_1X_5, X_1X_6,X_1X_7,X_2X_3X_4,X_2X_3X_5,X_2X_3X_6, X_2X_3X_7, \] \[X_2X_4X_5X_6,X_2X_4X_5X_7, X_3X_4X_5X_6X_7)\] of $K[X_1,\ldots, X_7]$. It is a squarefree lexsegment ideal with $\beta_{5,5+2}=1$, $\beta_{4,4+3}=1$, $\beta_{3,3+4}=1$, $\beta_{2,2+5}=1$ as extremal Betti numbers.\\ \\
For an USLI, we can state.
\begin{Prop} \label{exU} Let $I\subsetneq A$ be an USLI. Then $I$ has an unique extremal Betti number whose value is equal to $1$.
\end{Prop}
\begin{proof} Let $I$ be an USLI generated in degrees $d_ 1< d_2 < \ldots < d_t$. From Theorem \ref{char}, $\ell = R_{d_t} -1 = \max\{\m(u)\,:\,u\in G(I)_{d_t}\}$, where $d_t = \max\{j\,:\, G(I)_j\neq \emptyset\}$ and $\ell = \max\{\m(u)\,:\,u\in G(I)\}$. Hence, from Corollary \ref{value}, $\beta_{\ell - d_t,\,\ell - d_t+d_t}(I)$ is its unique extremal Betti number and its value is $1$.

\end{proof}
\begin{Cor}\label{inv} Let $I\subsetneq A$ be an USLI generated in degrees $d_ 1< d_2 < \ldots < d_t$. Then
\begin{enumerate}
\item $\proj_A(I) = \vert G(I)\vert -1$ and $\reg_A(I) = d_t$.
\item $\depth_A(A/I) = n - \vert G(I) \vert$.
\end{enumerate}
\end{Cor}
\begin{proof} (1). From Proposition \ref{exU}, $\beta_{\ell - d_t,\,\ell - d_t+d_t}(I)$, with $\ell =\max\{\m(u)\,:\,u\in G(I)_{d_t}\}$, is the unique extremal Betti number of $I$. Hence $\ell-d_t = \proj_A(I)$ and $d_t =\reg_A(I)$. With the same notations of Characterization \ref{char}, we have that
\[\proj_A(I) = \ell-d_t =  d_t + \sum_{i=1}^{t}\vert G(I)_{d_i} \vert- d_t-1= \vert G(I)\vert-1.\]
(2). It follows from the Auslander-Buchsbaum formula.
\end{proof}

\begin{Rem}\rm Recall that for a squarefree stable ideal $I\varsubsetneq A$, $\reg_A(I) = \max\{\deg u: $\newline $u \in G(I)\}$ \cite[Corollary 2.6]{Ahh:square}.
\end{Rem}



\section{USLIs, AUSLIs and $s$-sequences}\label{sec3}
In this section we study the strict link between USLIs (resp. ASLIs)
and $s$-sequences. We compute standard algebraic invariants of the
graded algebra $Sym_A (I)$ when $I$ is an USLI or an AUSLI of degree
$d$ in terms of the annihilator ideals of the $s$-sequence that
generates $I$.

Let $A$ be a noetherian ring, $M$ be a finitely generated $A$-module
with generators $f_1, \ldots, f_q$.  For every $i = 1, \dots , q$,
we set $M_{i-1} = Af_1 + \cdots + Af_{i-1}$ and let $I_i =
M_{i-1}:_A f_i$ be the colon ideal.  We set $I_0 = (0)$. Since
$M_i/M_{i-1} \simeq A/I_i$, so $I_i$ is the annihilator of the
cyclic module $A/I_i$. $I_i$ is called an annihilator ideal of the
sequence $f_1, \dots, f_q$.\\
Let $Sym_A(M)$ be the symmetric algebra of $M$.
 Let $(a_{ij})$, for $i=1,\ldots,q$, $j=1,\ldots,p$,  be
the relation matrix of $M$.  It is known that the symmetric algebra
$Sym_A(M)$ has a presentation $A[T_1,\ldots,T_q]/J$, with
$J=(g_1,\ldots,g_p)$
where $g_j=\sum_{i=1}^q a_{ij}T_i$ for $j=1, \ldots, p$.\\
We consider $S=A[T_1, \dots, T_q]$ a graded ring by assigning to
each variable $T_i$ degree $1$ and to the elements of $A$ degree
$0$. Then $J$ is a graded ideal and the natural epimorphism
$S\rightarrow Sym_A(M)$ is a homomorphism of graded $A$-algebras.\\
Let $<$ be a monomial order on the monomials in the variables
$T_i$ such that $T_1 < T_2  < \cdots < T_q$. With
respect to this term order, for any polynomial  $f = \sum a_{\alpha}
\underline{T}^{\alpha} \in S$, where $\underline{T}^{\alpha} =
T_1^{\alpha_1} \cdots T_q^{\alpha_q} $ and $\alpha=(\alpha_1,\ldots,
\alpha_q) \in \NN^q$, we put $ \textit{in}_{<} (f) =
a_{\alpha}\underline{T}^{\alpha}$, where $\underline{T}^{\alpha}$
is the largest monomial in $f$ such that $a_{\alpha} \neq 0$.\\
So we can define the monomial ideal
$\textit{in}_{<}(J)=(\textit{in}_{<}(f)| f\in J)$. Notice
that $(I_1T_1, I_2 T_2, \dots, I_q T_q)  \subseteq
\textit{in}_{<}(J)$ and the two ideals coincide in degree $1$.

\begin{Def}
The generators  $f_1, \dots, f_q$ of $M$ are called an $s$-sequence
(with respect to an admissible term order $<$) if
\[
(I_1T_1, I_2 T_2, \dots, I_q T_q) = in_{<}(J).
\]
If $I_1 \subseteq  I_2 \subseteq \cdots \subseteq  I_q$, the
sequence is a strong $s$-sequence.
\end{Def}

Now, let  $A = K[X_1,\dots, X_n]$ be the polynomial ring over a
field $K$ and let  $<$ any term order on
 $K[X_1,\dots, X_n;T_1,\dots, T_q]$ with $X_1  >
 \cdots > X_n$, $T_1 < T_2  < \cdots < T_q$, $X_i< T_j$ for all $i$ and  $j$. Then
 for any    Gr\"obner basis $G$ of  $J \varsubsetneq K[X_1, \dots, X_n; T_1,
\dots,T_q]$ with respect to $<$, we have $\textrm{in}_{<}(J)
=
 (\textrm{in}_{<} (f) |   f \in G)$. If the
elements of $G$ are of degree $1$ in the $T_i$, it follows that
$f_1, \dots, f_q$ is an $s$-sequence of $M$.

Let   $f_1,\ldots,f_q$ be monomials of $A$. Set
$f_{ij}=\displaystyle{\frac{f_i}{[f_i,f_j]}}$ for $i \neq j$, where
$[f_i,f_j]$ is the greatest common divisor of the monomials $f_i$
and $f_j$. $J$ is generated by $g_{ij}=f_{ij}T_j - f_{ji}T_i$ for
$1\leq i< j \leq q$. The  monomial sequence $f_1,\ldots,f_q$ is an
$s$-sequence if and only if $g_{ij}$, for $1\leq i< j \leq q$, is a
Gr\"{o}bner basis for $J$ for any term order which extends an
admissible term order on the $T_i$ in
 $S= K[X_1,\dots, X_n;T_1,\dots, T_q]$.
Note that the annihilator ideals of the monomial sequence
$f_1,\ldots,f_q$ are the ideals $I_i=(f_{1i},f_{2i}, \ldots,
f_{i-1,i})$ for $i=1,\ldots, q$.

\begin{Rem}
Let $I$ be an ideal of $A$ generated by an $s$-sequence $f_1,\dots,
f_q$ of monomials with respect to some admissible term order
$<$. From the theory of Gr\"obner bases,
one has that  $f_1, \dots, f_q$ is an  $s$-sequence with respect to
any other admissible term order (\cite{HRT}, Lemma 1.2).
\end{Rem}

For more details on this subject see \cite{HRT}.

\medskip

Since the property to be an  $s$-sequence may depend on the order on
the sequence, if $I$ is a squarefree lexicographic ideal when we write \[I=(f_1,f_2, \ldots, f_q),\] we suppose
\[f_1>_{\textrm{hslex}} f_2 >_{\textrm{hslex}} \cdots >_{\textrm{hslex}} f_q.\]
In order to simply the
notations we will denote $>_{\textrm{hslex}}$ by $>$. \\ For any positive integer $q$ we set $[q] = \{1, \ldots, q\}$.

The following lemma will be crucial in the sequel.
\begin{Lem}\label{L1}
Let $I=(f_1,\ldots,f_q)\varsubsetneq A$ be a squarefree lexsegment
ideal generated in degree $d$ with $G(I) \varsubsetneq M^s_d$.

The following conditions are equivalent:
\begin{enumerate}
\item  $[f_{ij},
f_{h\ell}]=1$, for $i<j$, $h<\ell$, $i\neq h$, $j\neq \ell$,
$i,j,h,\ell \in [q]$.
\item $|G(I)|\leq n-d+2$.
\item $I$ is an USLI or $I$ is an ASLI of $A$.
\end{enumerate}
\end{Lem}
\begin{proof} From Corollary \ref{samedeg} and Remark \ref{ASLI}, conditions (2) and (3) are equivalent, then we have only to prove that (1)$\Leftrightarrow$(2).\\
(1)$\Rightarrow$(2). Note that the monomial generators $f_i$ of $I$
are described by (\ref{degree2}), for $i=1, \ldots, n-d+2$. Suppose
$|G(I)|>n-d+2$. Since $I$ is a squarefree lexsegment ideal of degree
$d$, then $X_1X_2\cdots
X_{d-2}X_dX_{d+2}\in G(I)$. \\
From (\ref{degree2}), set $t = n-d+2$ and $t' = t+1$, then $f_t
=X_1X_2\cdots X_{d-2}X_dX_{d+1}$ and
$f_{t'} =  X_1X_2\cdots X_{d-2}X_dX_{d+2}$. Hence $f_{tt'} = X_{d+1}$. Again from (\ref{degree2}), $f_{23} =  X_{d+1}$. Hence $[f_{23}, f_{tt'}] = X_{d+1}$. A contradiction. \\
(2)$\Rightarrow$(1). Let $|G(I)|= q \leq n-d+2$ and $I =(f_1, \ldots, f_q)$.\\
If  $q < n-d+2$, then from (\ref{degree1}):
\[G(I) = \{X_1X_2\cdots
X_{d-1}X_d, X_1X_2\cdots X_{d-1}X_{d+1},\ldots, X_1X_2\cdots
X_{d-1}X_k\},\] with $d \leq k \leq n$. Hence:
\[f_{12}=X_{d},\,\,
f_{13}=X_{d},\,\, \ldots , \,\,f_{1q}=X_{d},\]
\[f_{23}=X_{d+1},\,\, \ldots,\,\,
f_{2q}=X_{d+1},\]
and so on.\\ By the structure of $G(I)$, this computation implies
that $f_{ij} \neq f_{h\ell}$ for $i\neq h$ and $j\neq \ell$. Hence
$[f_{ij}, f_{h\ell}]=1$ for $i<j$, $h<\ell$, $i\neq h$, $j\neq \ell$
and $i,j,h,\ell \in
[q]$. \\

If $q = n-d+2$, then from (\ref{degree2}):
\[G(I) = \{X_1X_2\cdots
X_{d-1}X_d, \ldots,
 X_1X_2\cdots X_{d-1}X_n, X_1X_2\cdots X_{d-2}X_dX_{d+1}\}.\] We have:
 \[ f_{12}=X_{d},\,\,
f_{13}=X_{d},\,\, \ldots, \,\,f_{1,n-d+1}=X_{d},\,\,
f_{1,n-d+2}=X_{d-1},\] \[f_{23}=X_{d+1},\,\, \ldots,\,\,
f_{2,n-d+1}=X_{d+1},\,\, f_{2,n-d+2}=X_{d-1}, \ldots,\,\,\]
\[
f_{1q}=X_{d-1}, \,\,f_{2q}=X_{d-1}, \,\,f_{3q}=X_{d-1}X_{d+2},\,\,
\ldots, \,\,f_{q-1,q}=X_{d-1}X_{n}.\] Hence $[f_{ij},f_{h\ell}]=1$
for $i<j$, $h<\ell$, $i\neq h$, $j\neq \ell$ and $i,j,h,\ell \in
[q]$.

\end{proof}

\begin{Thm}\label{T*}
Let $I\varsubsetneq A$ be a squarefree lexsegment ideal generated in
degree $d$ with $G(I) \varsubsetneq M^s_d$. Then $I$ is generated by
an $s$-sequence if and only if $|G(I)| \leq n-d+2$.
\end{Thm}

\begin{proof} Let $I=(f_1,f_2, \ldots, f_q)$ be a squarefree lexsegment ideal  and
suppose that $f_1,f_2, \ldots, f_q$ is an  $s$-sequence. We prove
that  $[f_{ij}, f_{h\ell}]=1$, for $i<j$, $h<\ell$, $i\neq h$,
$j\neq \ell$, with
$i,j,h,\ell \in [q]$. \\
The $s$-sequence property implies that $G=\{g_{ij}=f_{ij}T_j -
f_{ji}T_i \mid 1\leq i < j \leq q \}$ is a Gr\"{o}bner basis for
$J$. In particular, $S(g_{ij}, g_{h\ell})$ has a standard expression
with respect $G$ with remainder $0$. Note that, to get a standard
expression of $S(g_{ij}, g_{h\ell})$ is equivalent to find some
$g_{st} \in G$ whose initial term divides the initial term of
$S(g_{ij}, g_{h\ell})$ and substitute a multiple of $g_{st}$ such
that the remaindered polynomial has a  smaller initial term and so
on up to the remainder is $0$. We have:
$$
S(g_{ij}, g_{h\ell})= \frac{f_{ij}f_{\ell
h}}{[f_{ij},f_{h\ell}]}T_jT_h -
\frac{f_{h\ell}f_{ji}}{[f_{ij},f_{h\ell}]}T_iT_\ell.$$ First observe
that $[f_{ij},f_{j\ell}]=1$ as $f_1, \ldots, f_q$ are squarefree
monomials. \\ Now we consider the other cases. Suppose that $i<j$,
$h<\ell$, $i\neq h$, $j\neq \ell$. As $S(g_{ij}, g_{h\ell})$ has a
standard expression with respect $G$ there exists $g_{st}$ such that
$in_<(g_{st})$ divides
$in_< (S(g_{ij}, g_{h\ell}))$.\\
We distinguish two cases: $\ell>j$, $\ell<j$.\\
Let $\ell>j$, then $in_<(g_{st}) \mid
\frac{f_{h\ell}f_{ji}}{[f_{ij},f_{h\ell}]}$.\\
If $f_{h\ell} \mid \frac{f_{h\ell}f_{ji}}{[f_{ij},f_{h\ell}]}$, then
$[f_{ij},f_{h\ell}] \mid f_{ji}$. But, since $[[f_{ij},f_{h\ell}],
f_{ji}]=1$, it follows $[f_{ij},f_{h\ell}]=1$.\\
Consider $f_{s\ell} \mid
\frac{f_{h\ell}f_{ji}}{[f_{ij},f_{h\ell}]}$, where
$f_{s\ell}=in_<(g_{s\ell})$ with $s<j$ and $s<h$. We can write:
$$
S(g_{ij}, g_{h\ell})= -
\frac{f_{ji}f_{h\ell}}{f_{s\ell}[f_{ij},f_{h\ell}]}g_{s\ell} T_i+
\frac{f_{ij}f_{\ell h}}{[f_{ij},f_{h\ell}]}T_jT_h -
\frac{f_{ji}f_{h\ell}f_{\ell
s}}{f_{s\ell}[f_{ij},f_{h\ell}]}T_iT_s.$$ Hence $\frac{f_{ij}f_{\ell
h}}{[f_{ij},f_{h\ell}]}T_jT_h$ is divided by $f_{ij}$ and
consequently $[f_{ij},f_{h\ell}] \mid f_{\ell h}$. But, since
$[[f_{ij},f_{h\ell}],
f_{\ell h}]=1$, then it follows $[f_{ij},f_{h\ell}]=1$.\\
Let $\ell<j$, then $in_<(g_{st}) \mid
\frac{f_{ij}f_{\ell h}}{[f_{ij},f_{h\ell}]}$.\\
If $f_{ij} \mid \frac{f_{ij}f_{\ell h}}{[f_{ij},f_{h\ell}]}$, then
$[f_{ij},f_{h\ell}] \mid f_{\ell h}$. But, as we have
$[[f_{ij},f_{h\ell}],
f_{\ell h}]=1$, then it follows $[f_{ij},f_{h\ell}]=1$.\\
Consider $f_{sh} \mid \frac{f_{ij}f_{\ell h}}{[f_{ij},f_{h\ell}]}$,
where $f_{sh}=in_<(g_{sh})$. We can write:
$$
S(g_{ij}, g_{h\ell})= \frac{f_{ij}f_{\ell
h}}{f_{sh}[f_{ij},f_{h\ell}]}T_jg_{sh} -
\frac{f_{h\ell}f_{ji}}{[f_{ij},f_{h\ell}]}T_iT_{\ell} +
\frac{f_{ij}f_{\ell h}f_{hs}}{f_{sh}[f_{ij},f_{h\ell}]}T_jT_s.$$
Hence $\frac{f_{ij}f_{\ell
h}f_{hs}}{f_{sh}[f_{ij},f_{h\ell}]}T_jT_s$ is divided by $f_{ij}$.
Therefore $f_{sh}[f_{ij},f_{h\ell}] \mid f_{\ell h}f_{hs}$. But as
we have $[[f_{ij},f_{h\ell}], f_{\ell h}]=1$, then $f_{sh} \mid
f_{\ell h}$ and $[f_{ij},f_{h\ell}] \mid f_{hs}$. By the structure
of the monomials $f_1, \ldots, f_q$, if $[f_{ij},f_{h\ell}] \mid
f_{hs}$, with $s<h$, then $[f_{ij},f_{h\ell}]=1$.\\ Hence in any
case  we have $[f_{ij},f_{h\ell}]=1$ for $i<j$, $h<\ell$, $i\neq h$,
$j\neq \ell$, with  $i,j,h,\ell \in [q]$. It follows $|G(I)|=q
\leq n-d+2$ by Lemma \ref{L1}.\\
\\
Now suppose $|G(I)| \leq n-d+2$. Hence from Lemma \ref{L1},
$[f_{ij}, f_{h\ell}]=1$, for $i<j$, $h<\ell$, $i\neq h$, $j\neq
\ell$, $i,j,h,\ell \in [q]$ and the assert follows from \cite{HRT}
(Proposition 1.7).
\end{proof}

\begin{Expl}  $A=K[X_1,X_2,X_3,X_4]$, $I=(X_1X_2, X_1X_3, X_1X_4, X_2X_3,
X_2X_4)$.\\
Set $f_1=X_1X_2$, $f_2=X_1X_3$,  $f_3=X_1X_4$, $f_4=X_2X_3$,
$f_5=X_2X_4$.\\ $G=\{f_{ij}T_j - f_{ji}T_i | \ \ 1\leq i< j \leq 5
\}$ is not a Gr\"{o}bner basis for $J$. In fact $J$ does not admit a
linear Gr\"{o}bner basis $BG(J)$ for any term order in $A[T_1,
\ldots,T_5]$: $BG(J)= \{X_2T_4 - X_3T_5, X_1T_2 - X_3T_5, X_2T_3 -
X_4T_5, X_1T_1 - X_4T_5, X_3T_3 - X_4T_4,$\\ $X_3T_1 - X_4T_2,
X_4T_2T_3 - X_4T_1T_4\}$ \cite{CNR}. Hence $f_1, \ldots, f_5$ is not an
$s$-sequence.\\
\end{Expl}

\begin{Rem}\rm{If $G(I) = M^s_d$, then $I=\mathcal{I}_d$, $2\leq d \leq n$, where $\mathcal{I}_d$ is the
Veronese  ideal of $A$ generated by all the squarefree monomials of
degree $d$ in the variables $X_1,\dots, X_n$. The ideal
$\mathcal{I}_d$ is generated by an $s$-sequence if and only if
$d=n-1$ (\cite{LR}, Theorem 2.3). Hence $I=(M^s_d)$ is generated by
an $s$-sequence if and only if $d=n-1$.}\\
\end{Rem}

\begin{Prop}\label{P1}
Let $I\varsubsetneq A$ be a squarefree lexsegment ideal generated in
degree $d$ such that $|G(I)|\leq n-d+2$. Then the annihilator ideals
of the sequence of the monomial generators of $I$ are:$$ I_1=(0), \
\ I_i=(X_d, \ldots, X_{d+i-2}) \ \ for \ \ i=2, \ldots,n-d+1,$$
$$I_i=(X_{d-1}) \ \ for \ \ i=n-d+2.$$
\end{Prop}
\begin{proof}  Set $|G(I)|=q$. Let $I= (f_1,\ldots,f_q)$  with $f_1> \ldots
> f_q$. \\
Set $f_{ij}= \displaystyle{\frac{f_i}{[f_i,f_j]}}$ for $i < j$,
$i,j\in [q]$. Then the annihilator ideals of the monomial sequence
$f_1,\ldots,f_q$ are $I_i=(f_{1i},f_{2i}, \ldots, f_{i-1,i})$, for
$i\in [q]$. \\ For $i=1$,  $I_1=(0)$ and by the  structure of these
monomials, it follows:
\[I_2=(f_{12})=(X_{d}),\,\, I_3=(f_{13},
f_{23})=(X_{d}, X_{d+1}),\,\, \ldots,\,\,\]
\[I_{n-d+1}=(f_{1,n-d+1},f_{2,n-d+1}, \ldots,f_{n-d,n-d+1})=(X_{d},
X_{d+1}, \ldots, X_{n-1})\] and \[I_{n-d+2}=(f_{1,n-d+2},  \ldots,
f_{n-d+1,n-d+2})=(X_{d-1},X_{d-1}X_{d-2},\ldots,
X_{d-1}X_{n})=(X_{d-1}).\] Hence the assert follows. \end{proof}

\begin{Rem}
If $I$ is an USLI of $A$ generated in degree $d$, then $I$ is generated by a strong
$s$-sequence. In fact $I_1 \varsubsetneq I_2 \varsubsetneq \cdots
\varsubsetneq I_{n-d+1}$ for $|G(I)|< n-d+2$.
\end{Rem}

\begin{Thm}\label{inv:USLI}
Let $I\varsubsetneq A$ be an USLI generated in degree $d$. Then
\begin{enumerate}
\item $\dim (Sym_A(I)) = n+1$;
\item  $\e(Sym_A(I))=|G(I)|$;
\item $\reg_A (Sym_A(I))= 1$;
\item $\depth_A (Sym_A(I)) = n+1$.
\end{enumerate}
\end{Thm}
\begin{proof}
(1) $I$ is generated by  a strong $s$-sequence. Hence, by \cite{T}
(Theorem 4.8), $Sym_A(I)$ has  dimension $\dim (A) + 1=n+1$.\\
(2) Let $|G(I)|=q$. By \cite{HRT} (Proposition 2.4) we have
$\e(Sym_A(I))= \sum_{i=1}^q \e(A/I_{i})$. By Proposition \ref{P1},
the annihilator ideals $I_i$ are generated by a regular sequence,
then, by \cite{T} (Theorem 4.8), $\e(A/I_i)=1$, for $i=2,\ldots,q$
and $\e(A/(0))=1$. Hence
 $\e(Sym_A(I))=\sum_{i=1}^q\e(A/I_i)=q$.\\
(3) By \cite{T} (Theorem 4.8):
\begin{eqnarray*}
  \reg (Sym_A(I)) &=& \reg_A (A[T_1,\ldots,T_{q}]/J) \\
   &\leq &  \reg_A (A[T_1,\ldots,T_{q}]/in_{<}(J)) = \reg_A (A[T_1,\ldots,T_{q}]/(I_1T_1,\ldots, I_qT_q)) \\
   &\leq & \max_{2\leq i\leq q} \{\sum_{j=1}^{i-1} \deg X_j- (i-2)\}=(i-1)-(i-2)=1,
\end{eqnarray*}
for $i=2, \ldots, q$.
Since $J$ is generated by the linear forms of degree two $X_iT_j -
X_jT_i$, for $i,j=1, \ldots,q$, then $\reg_A (A[T_1,\ldots,T_{q}]/J)
\geq 1$. It follows that
$\reg_A (Sym_A(I))=1$.\\
(4) $\depth_A (Sym_A(I))\geq \dim(A) +1 = n+1$ (\cite{T}, Theorem 4.8).\\
Since $\depth_A (Sym_A(I)) \leq \dim Sym_A(I)) = n+1$, the assert
follows.
\end{proof}

\begin{Rem}
If $I$ is an USLI of $A$ generated in degree $d$, then $Sym_A(I)$ is Cohen-Macaulay.
\end{Rem}

\begin{Thm} \label{inv:ASLI}
Let $I\varsubsetneq A$ be an AUSLI generated in degree $d$. Then
\begin{enumerate}
\item $\dim (Sym_A(I)) = n+1$.
\item $\e(Sym_A(I))=2|G(I)| - 2$.
\end{enumerate}
\end{Thm}
\begin{proof} 
(1) By  \cite{HRT} (Proposition 2.4), we have $\dim (Sym_A(I))=
\textrm{max} \{ \dim (A/(I_{i_1} + \ldots + I_{i_r})) + r \},$ for
$1\leq i_1 < \cdots < i_r \leq n-d+2$. \\Hence, by Proposition
\ref{P1}: \[\dim (Sym_A(I))=  (n-r+ 1) + r = n+1. \]
(2) We have that $|G(I)|=n-d+2$.\\ From \cite{HRT} (Proposition 2.4),
  \[\e(Sym_A(I))= \sum_{1\leq i_1 < \cdots < i_r
\leq n-d+2} \e(A/(I_{i_1}+\ldots+ I_{i_r}))\]  with $\dim
(A/(I_{i_1}+\ldots + I_{i_r}))=\dim (Sym_A(I))-r = n+1-r$,  $1\leq r \leq n-d+2$.\\
By Proposition \ref{P1}, $I_i=(X_d, \ldots, X_{d+i-2})$, for $i=2,
\ldots,n-d+1$,
$I_i=(X_{d-1})$ for $i=n-d+2$. \\
Set  $H=I_{i_1}+\ldots+ I_{i_r}$. Hence $A/H$ is Cohen-Macaulay and
has a linear resolution with projective dimension equal to the
number of the generators of $H$ (\cite{HK}). Then $e(A/H)=1$, by
Huneke-Miller formula (\cite{HM}).\\
Set $|G(I)|=n-d+2=q$ and $d'=\dim (A/H)=n+1-r$. Then $\e(Sym_A(I))$
is given by the sum of the following
terms:\\
$\e(A/(0))=1$, for $r=1$ and $d'=n$,\\
$\e(A/(I_1+I_2))=1$, $\e(A/(I_1+I_q))=1$, for $r=2$ and $d'=n-1$,\\
$\e(A/(I_{1}+ I_{2}+I_3))=1$, $\e(A/(I_{1}+ I_{2}+I_q))=1$, for $r=3$ and $d'=n-2$,\\
$\e(A/(I_{1}+I_{2}+I_{3}+I_4))=1$, $\e(A/(I_{1}+I_{2}+I_{3}+I_q))=1$, for $r=4$ and $d'=n-3$,\\
and so on up to\\
$\e(A/(I_{1}+I_{2}+\ldots + I_{q-1}))=1$,  $\e(A/(I_{1}+\ldots +I_{q-2}+I_q))=1$, for $r=n-d+1$ and $d'=d-2$,\\
$\e(A/(I_1+I_2+\cdots +I_q))=1$, for $r=n-d+2$ and $d'=d-1$.\\
 Hence
 \begin{eqnarray*}
   \e(Sym_A(I)) &=& e(A/(0)) + 2 e(A/(I_1+I_2))+ \cdots + 2
e(A/(I_{1}+I_{2}+\ldots + I_{q-1}))\\
    &&+  e(A/(I_1+I_2+ \cdots +I_q))=
2(q-2)+2=2q-q.
 \end{eqnarray*}
\end{proof}


\end{document}